\documentclass[11pt]{article} 

\usepackage{bbm}
\usepackage[utf8]{inputenc} 
\usepackage{enumerate,amsmath,amssymb,amsthm,bbm}
\usepackage{hyperref}
\usepackage{graphicx}
\usepackage[]{algorithm2e}
\usepackage[utf8]{inputenc}
\usepackage[L7x]{fontenc}
\usepackage[english]{babel}
\usepackage{amssymb, amsmath, amsthm}
\usepackage{enumerate}
\usepackage{hyperref}
\usepackage{a4wide}
\usepackage[usenames]{color}
\usepackage{soul}
\usepackage[all]{xy}
\usepackage{array}
\usepackage{tabularx}
\usepackage[table]{xcolor}
\usepackage{booktabs}
\usepackage{verbatim}
\usepackage{rotating}
\usepackage{xtab}
\usepackage{mathtools}
\usepackage{amsmath}
\usepackage{lmodern}
\usepackage{caption}
\usepackage{multirow}
\usepackage{float}
\usepackage{mdwlist}
\usepackage{breqn}
\usepackage{amsmath}
\usepackage{amsfonts}
\usepackage{setspace}

\usepackage{geometry} 
\usepackage{graphicx} 

\usepackage[yyyymmdd]{datetime}
\usepackage{booktabs} 
\usepackage{array} 
\usepackage{paralist} 
\usepackage{verbatim} 
\usepackage{subfig} 

\usepackage{fancyhdr} 
\usepackage{sectsty}
\allsectionsfont{\sffamily\mdseries\upshape} 

\usepackage[nottoc,notlof,notlot]{tocbibind} 
\usepackage[titles,subfigure]{tocloft} 

\theoremstyle{plain}
\newtheorem{theorem}{Theorem}
\newtheorem{lemma}{Lemma}

\newtheorem{proposition}{Proposition}

\newtheorem{corollary}{Corollary}
\theoremstyle{remark}

\title{On the Congruence of Finite Generalized Harmonic Numbers Sums Modulo $p^2$}
\author{Aidas Medžiūnas}

\begin{document}
\maketitle
\begin{abstract}
In this paper we investigate congruence relationships of particular finite generalized harmonic numbers sums. We suggest more transparent and simpler method to analyse these sums and present several additional results for certain special cases.
\end{abstract}
\section{Introduction}

We will define the harmonic number as the sum of the reciprocals of the first $n$ natural numbers: 
\begin{equation*}\label{eq:harmonicNumber}
	H_{n}=\sum_{k=1}^n\dfrac{1}{k},
\end{equation*}
Investigation of the harmonic numbers properties traces its history back to Ancient Greece and has the fundamental importance to the several fields of mathematics. These numbers are closely related to Riemann zeta function and appear in many expressions of other special functions \cite{Choi}.   
Naturally the definition of harmonic numbers can be expanded to the sums of reciprocal powers. Hence, we will denote generalized harmonic numbers (GHN) as
\begin{equation*}\label{eq:genHarmonicNumber}
	H_{n}^{(r)}:=\sum_{k=1}^n\dfrac{1}{k^r},
\end{equation*}
where $r=\sigma+it$ is a complex variable. However, in the study of divisibility it is assumed that $r$ is an integer.

Consider the behaviour of the finite sums of terms involving GHN in the form 
\begin{equation*}
	S_{a,b,k,p}:=\sum_{i=1}^{p-1}i^b\left(H_{i}^{(a)}\right)^k,
\end{equation*}
where $a,b,k\in\mathbb{Z}$ and $p$ is a prime number. For simplicity of notation, we write $S_{a,b,k}$ instead of $S_{a,b,k,p}$. These sums demonstrate interesting results modulo $p$, $p^2$, $p^3$, yet no universal theory for all $a,b,k$ values exists. Therefore one has to approach the study of this mathematical object from variety of narrower, more specific directions.

The number of authors investigated cases for the particular coefficients $a,b,k$. For example, Meštrović \cite{Mestrovic} generalised the problem, conjectured by Tauraso \cite{Tauraso} and proven by Tyler \cite{Tyler}, elegantly demonstrating that
\begin{equation*}
	S_{1,-2,1}\equiv S_{1,-1,2}\equiv-\dfrac{3}{p^2}H_{p-1}\equiv\dfrac{3}{2p}H^{(2)}_{p-1}\pmod{p^2}.
\end{equation*}
In his other work Meštrović with Andjić \cite{Mestrovic2} following the work done by Sun \cite{Sun} showed that for the following binomial sums of harmonic numbers, when $p>3$, we get
\begin{equation*}
	\sum_{j=m}^{p-1}\binom{j}{m}H_{j}\equiv\dfrac{(-1)^m}{m+1}\left(1-pH_{m+1}+\dfrac{p^2}{2}\left(H^2_{m+1}-H_{m+1}^{(2)}\right)\right)\pmod{p^3}
\end{equation*}
and therefore (see Corollary 1.3\cite{Mestrovic2})
\begin{align*}
	S_{1,1,1}&\equiv-\dfrac{p^2-3p+2}{4}\pmod{p^3},\\
	S_{1,2,1}&\equiv\dfrac{15p^2-17p+6}{36}\pmod{p^3},\\
	S_{1,3,1}&\equiv-\dfrac{21p^2-10p}{48}\pmod{p^3}.
\end{align*}

Other authors looked for more general ways to describe the behaviour of these sums, when one or two coefficients are not specified and showed that the connection between these sums and the Bernoulli numbers exists. For example, Zhao \cite{Zhao} using the properties of multiple harmonic sums showed that for two positive integers $a$ and $b$, which are of the same parity such that $p>a+b+1$, we get
\begin{align}\label{eq:Zhao}
	S_{a,-b,1}&\equiv p\left(k_{a,b}+a+b\right)\dfrac{B_{p-a-b-1}}{2(a+b+1)}\pmod{p^2},
\end{align}
where $k_{a,b}=(-1)^a\left[\binom{a+b+1}{b}-\binom{a+b+1}{a}\right]$ and $B$ are the Bernoulli numbers.

Another interesting result in this field is Conjecture 1.2 proposed by Sun \cite{Sun} and proved by Sun and Zhao \cite{Sun2}, which states that for any odd prime $p$ and a positive integer such that $p-1\nmid6a$, we get
\begin{equation}\label{eq:lyginiaiDalumai}
	S_{2a,-2a,2}\equiv 0\pmod{p}.
\end{equation}
In this paper we are concerned with the results recently obtained by Wang and Yang \cite{kinieciaip2}. They showed that
\begin{align}\label{eq:kinieciai1}
	S_{1,b,1}&\equiv(1-p)B_b-\dfrac{p}{b+1}\sum_{j=0}^b\binom{b+1}{j}B_jB_{b-j}\pmod{p^2}
\end{align}
and
\begin{align}\label{eq:kinieciai2}
	S_{1,b,2}&\equiv B_{b-1}-\dfrac{p(b+2)}{2b}\sum_{j=0}^{b-1}\binom{b}{j}B_jB_{b-j-1}+\dfrac p4\sum_{j=0}^{b-1}B_jB_{b-j-1}\pmod{p^2},	
\end{align}
where $b$ is an odd positive integer such that $3<b<p-1$.

In our work we will (by more transparent and simpler means) prove more general result for $S_{a,b,1}$ and $S_{a,b,2}$. Furthermore, our approach will provide several new results for particular $a,b,k$, complementing previous works in the field found in Sections \ref{auxi} and \ref{particular}.
\section{Main results}
In this section we present our main results about the congruence of $S_{a,b,1}$ and $S_{a,b,2}$.
\begin{theorem}\label{th:main1}
	Let $a,b$ be non-negative integers and $p$ a prime number such that $\max(a,b)+3\leq p$. Then, the following congruences are true:
\begin{enumerate}
\item
if $0<a\leq b$
\begin{align*}
	S_{a,b,1}\equiv& \dfrac{1}{b+1}\binom{b+1}{a}B_{b-a+1}-\dfrac{p}{b+1}\sum_{\substack{j=0}}^{b-a+1}\binom{b+1}{j}B_{b-a+1-j}B_j\\&-\dfrac{p}{b+1}\sum_{\substack{j=b-a+2\\a\geq2}}^{b}\binom{b+1}{j}B^\#_{a,b,j,p}\pmod{p^2},
\end{align*}
\item
if $0<a=b+1$
\begin{align*}
	S_{a,b,1}\equiv& \dfrac{1-p}{a}-\dfrac pa\sum_{\substack{j=1}}^{a-1}\binom{a}{j}B^\#_{a,a-1,j,p}\pmod{p^2},
\end{align*}
\item
if $0<b+1<a$
\begin{align*}
	S_{a,b,1}\equiv& -\dfrac{p}{b+1}\sum_{\substack{j=0}}^{b}\binom{b+1}{j}B^\#_{a,b,j,p}\pmod{p^2},
\end{align*}
\end{enumerate}
where $B^\#_{a,b,j,p}=\dfrac{a-b-1+j}{a-b+j}B_{p-a+b-j}B_j$.
\end{theorem}

Note that, taking $a=1$ in Theorem \ref{th:main1} we get Eq. (\ref{eq:kinieciai1}). Furthermore, for more explicit expressions of $S_{a,b,1}$, depending on the parity of $a,b$ see Corollaries \ref{col:nelyginiai} and \ref{col:lyginiai}.
\begin{theorem}\label{th:main2}
	For positive integers $a,b$, such that $2\leq a<b$, $2\nmid b$ and a prime number $p:\max(2a,b)+3\leq p$ the following congruences are true:
\begin{enumerate}
\item
	if $b+1<2a$
\begin{align*}
	S_{a,b,2}\equiv&\dfrac{p}{b-a+1}\binom{b}{a}B^\#_{2a,b,b-a+1,p}-\dfrac{p}{b-a+1}\sum_{\substack{j=0}}^{b-a}\binom{b-a+1}{j}B^\#_{2a,b,j,p}\\&-\dfrac{p}{a(a+1)}\sum_{j=b-a+2}^b\binom{b-a-j-1}{a-1}\binom{b}{j}B_jB_{p-2a+b-j}\pmod{p^2},
\end{align*}
\item
if $b+1=2a$
\begin{align*}
	S_{a,b,2}\equiv&\dfrac{1-p}{a}-\dfrac{p}{2a(a+1)}-\dfrac{p}{a}\sum_{\substack{j=1}}^{a-1}\binom{a}{j}B^\#_{2a,2a-1,j,p}\\&-\dfrac{p}{a(a+1)}\sum_{j=a+1}^{2a-1}\binom{a-j-2}{a-1}\binom{2a-1}{j}B_jB_{p-j+1}\pmod{p^2},
\end{align*}
\item
if $b+1>2a$, 
\begin{align*}
S_{a,b,2}\equiv&\dfrac{1-p}{b-a+1}\binom{b-a+1}{a}B_{b-2a+1}-\dfrac{p}{b-a+1}\sum_{j=0}^{b-a}\binom{b-a+1}{j}B^\#_{2a,b,j,p}\\&-\dfrac{p}{(a+1)(b+1)}\sum_{\substack{j=0}}^{b-2a+1}\binom{b+1}{j}\binom{b-a+1-j}{a}B_{b-2a+1-j}B_{j}\\&-\dfrac{p}{a(a+1)}\sum_{j=b-a+1}^b\binom{b-a-j-1}{a-1}\binom{b}{j}B_jB_{p-2a+b-j}\pmod{p^2}.
\end{align*}
\end{enumerate}
\end{theorem}

The proofs of Theorems \ref{th:main1} and \ref{th:main2} are provided in Sections \ref{proof1} and \ref{proof2}.
\section{Auxiliary results}\label{auxi}
In this section we will form several auxiliary results about GHM sums and later about their congruence.
\subsection{Two and three dimensional identities involving sums of GHN}
\begin{proposition}\label{prop:2d}The following identity is true
\begin{align}\label{eq:a+b}	
		\sum_{i=1}^{n}i^{-a}H_i^{(b)}+\sum_{i=1}^{n}i^{-b}H_i^{(a)}&=H_{n}^{(a)}H_{n}^{(b)}+H_{n}^{(a+b)},
	\end{align}
	where $a,b\in\mathbb R$.
\end{proposition}
\begin{proof}
	Note that
	\begin{align*}			
		H_{n}^{(a)}H_{n}^{(b)}&=\sum_{\substack{k,l=1}}^{n}k^{-a}l^{-b}=\sum_{\substack{k,l=1 \\ k\geq l}}^{n}k^{-a}l^{-b}+\sum_{\substack{k,l=1 \\ k\leq l}}^{n}k^{-a}l^{-b}-\sum_{\substack{k,l=1 \\ k=l}}^{n}k^{-a}l^{-b}=\\&=\sum_{\substack{i=1}}^{n}i^{-a}H_i^{(b)}+\sum_{\substack{i=1}}^{n}i^{-b}H_i^{(a)}-H_n^{(a+b)}.
	\end{align*}	
\end{proof}

Equation (\ref{eq:a+b}) and its more specific counterpart
\begin{equation}
	\sum_{i=1}^{n}i^{-a}H_i^{(a)}=\dfrac12\left(\left(H_{n}^{(a)}\right)^2+H_{n}^{(2a)}\right)
\end{equation}
are the more general case formulas for the identity presented by Alzer et al. [\cite{Alzer},Eq.(3.62)]
\begin{align}			
		\sum_{i=1}^{n}\dfrac{H_i}{i}&=\dfrac{1}{2}\left(H^2_{n}+H_{n}^{(2)}\right).
\end{align}
Furthermore, taking $b=0$ gives us a generalization
\begin{align}\label{eq:b=0}
		\sum_{i=1}^{n}H_i^{(a)}=(n+1)H_{n}^{(a)}-H_{n}^{(a-1)}
\end{align} 
for another well known formula
\begin{align}			
		\sum_{i=1}^{n}H_i=(n+1)H_{n}-n.
\end{align} 

By using Eq.\ref{eq:a+b} in conjunction with the following elementary identity 
$$H_{n+1}^{(a)}=H_{n}^{(a)}+(n+1)^{-a},$$
we obtain
\begin{align}\label{eq:i+1}	
		\sum_{i=1}^{n}(i+1)^{-a}H_i^{(b)}+\sum_{i=1}^{n}(i+1)^{-b}H_i^{(a)}&=H_{n+1}^{(a)}H_{n+1}^{(b)}-H_{n+1}^{(a+b)},
\end{align}
which is generalized case of equations presented by Choi and Srivastava [\cite{Choi},Eq.(1.31)]
\begin{align}
		\sum_{i=1}^{n}\dfrac{H_i}{i+1}&=\dfrac12\left(H^2_{n+1}-H_{n+1}^{(2)}\right).
\end{align}
For the negative integer coefficients of GHN, Faulhaber's formula can be applied
\begin{align*}	
	H_{n}^{(-a)}=\dfrac{1}{a+1}\sum_{j=0}^{a}\binom{a+1}{j}(-1)^jB_jn^{a+1-j}.
\end{align*}
Using this fact in conjunction with (\ref{eq:a+b}) we obtain that, for example
\begin{equation}\label{eq:b1}
	2\sum_{i=1}^{n}iH_i^{(a)}=n(n+1)H_{n}^{(a)}+H_n^{(a-1)}-H_n^{(a-2)}.
\end{equation}	
\begin{proposition}\label{prop:a+b+c}The following identity is true
\begin{gather}
\begin{aligned}\label{eq:a+b+c}	
		H_{n}^{(a)}H_{n}^{(b)}H_{n}^{(c)}-H_{n}^{(a+b+c)}=&\sum_{i=1}^{n}i^{-a}H_i^{(b)}H_i^{(c)}+\sum_{i=1}^{n}i^{-b}H_i^{(a)}H_i^{(c)}+\sum_{i=1}^{n}i^{-c}H_i^{(a)}H_i^{(b)}\\&-\left(\sum_{i=1}^{n}i^{-a-b}H_i^{(c)}+\sum_{i=1}^{n}i^{-b-c}H_i^{(a)}+\sum_{i=1}^{n}i^{-a-c}H_i^{(b)}\right),
\end{aligned}
\end{gather}
	where $a,b,c\in\mathbb R$.
\end{proposition}

This is a straightforward three dimensional generalisation of Proposition \ref{eq:a+b}.
It has to be mentioned that similar identity (only investigating multiple harmonic sums) was achieved by Kehila \cite{Kehila}.
\subsection{General case congruences}
We put that
$$\dfrac mn\equiv \dfrac rs \pmod{p}\Longleftrightarrow ms\equiv nr\pmod{p},$$
where $n,s$ are not divisible by $p$.

In 1900, Glaisher \cite{Glaisher} (see also Eq.22\cite{Mestrovic}) proved that:
\begin{lemma}\label{lema:dalumasBernulio}
Let $m$ be a positive integer,
and let $p$ be a prime such that $p\geq m+3$. Then
\begin{equation*}
H_{p-1}^{(m)}\equiv
\begin{cases}
\dfrac{m}{m+1}pB_{p-m-1}\pmod{p^2}\quad\text{if $m$ is even,}\\
\dfrac{m(m+1)}{2(m+2)}p^2B_{p-m-2}\pmod{p^3}\quad\text{if $m$ is odd.}
\end{cases}
\end{equation*}
\end{lemma}

Proposition \ref{prop:a+b+c} (case $a=b=c$), when $p\geq3a+3$, in conjunction with Glaisher's lemma generalizes the congruence by Tauraso \cite{Tauraso} mentioned earlier in the paper:
\begin{equation}\label{eq:hyp}
		\sum_{i=1}^{p-1}i^{-a}\left(H_i^{(a)}\right)^2\equiv\sum_{i=1}^{p-1}i^{-2a}H_i^{(a)}\equiv-\sum_{i=1}^{p-1}i^{-a}H_i^{(2a)}
\begin{cases}
\pmod{p}\quad\text{if $a$ is even,}\\
\pmod{p^2}\quad\text{if $a$ is odd.}
\end{cases}
\end{equation}	
Furthermore combining other cases of Proposition \ref{prop:a+b+c} ($c=0$; $a=b,\,c=0$) and (\ref{eq:Zhao}) we get several new generalizations:
\begin{equation}\label{eq:a2}
		\sum_{i=1}^{p-1}\left(H_i^{(a)}\right)^2\equiv-2\sum_{i=1}^{p-1}i^{-a+1}H_i^{(a)}\pmod{p^2},
\end{equation}
when $p\geq2a+2$;
\begin{equation}
		\sum_{i=1}^{p-1}H_i^{(a)}H_i^{(b)}\equiv-\dfrac{k_{a,b}}{2(a+b)}pB_{p-a-b}\pmod{p^2},
\end{equation}
when  $a$ and $b$ are of the opposite parity and $p\geq a+b+2$. These generalise results given by Sun \cite{Sun} for particular values of $a,b,k$.
\section{Proof of main results}
For the proof of Theorems \ref{th:main1} and \ref{th:main2} we are going to need the following results.
\begin{theorem}(Von Staudt–Clausen)
If $n$ is a positive integer then
$$B_{2n}+\sum_{(p-1)|2n}\dfrac 1p\in\mathbb{Z}.$$
\end{theorem}

As a consequence of this theorem we have that the denominator of $B_{2n}$ is not divisible by $p\geq 2n+3$.
\begin{lemma}\label{lema:help1}
	Assume that $4\leq a+3\leq p$. Then
	$$H^{(-a)}_{p-1}\equiv pB_a\pmod{p^2}.$$
\end{lemma}
\begin{proof}
	From Faulhaber's and Newton's binomial formulas follows:
	\begin{align*}
		H^{(-a)}_{p-1}&=\dfrac{1}{a+1}\sum_{j=0}^a\binom{a+1}{j}(-1)^jB_j\sum_{k=0}^{a+1-j}\binom{a+1-j}{k}p^k(-1)^{a+1-j-k}\\&\equiv\dfrac{(-1)^a}{a+1}\sum_{j=0}^a\binom{a+1}{j}B_j\left[\left(a+1-j\right)p-1\right]\\&=(-1)^ap\sum_{j=0}^a\binom{a}{j}B_j-\dfrac{(-1)^a}{a+1}\delta_{a,0}\pmod{p^2},
	\end{align*}
	where $\delta_{a,0}$ is Kronecker delta.
\end{proof}

Immediately, we see, that due to properties of Bernoulli numbers $H^{(-a)}_{p-1}$ is divisible by $p^2$, when $a>1$ is odd. 
\subsection{Proof of Theorem 1}\label{proof1}
We begin by deducing from (\ref{eq:a+b}) that
\begin{align*}
	\sum_{i=1}^{p-1}i^{b}H_i^{(a)}&=H_{p-1}^{(-b)}H_{p-1}^{(a)}+H_{p-1}^{(a-b)}-\sum_{i=1}^{p-1}i^{-a}\sum_{j=1}^ij^b.
\end{align*}
It is well known that:
\begin{lemma}\label{lema:minusdalumas}
	If $m\in\mathbb{N}$ and $p$ is prime, then
\begin{equation*}
H_{p-1}^{(-m)}\equiv
\begin{cases}
-1 \pmod{p}\quad\text{if $p-1\mid m$,}\\
0 \pmod{p}\quad\text{if $p-1\nmid m$.}\\
\end{cases}
\end{equation*}
\end{lemma}
Hence, by Faulhaber's formula and Lemmas \ref{lema:dalumasBernulio} and \ref{lema:minusdalumas} we get
\begin{gather*}
\begin{aligned}\label{eq:main10}
	\sum_{i=1}^{p-1}i^{b}H_i^{(a)}&\equiv H_{p-1}^{(a-b)}-\dfrac{1}{b+1}\sum_{\substack{j=0}}^b\binom{b+1}{j}(-1)^jB_jH_{p-1}^{(a-b-1+j)}\\&=-\dfrac{1}{b+1}\sum_{\substack{j=0}}^b\binom{b+1}{j}B_jH_{p-1}^{(a-b-1+j)}\pmod{p^2}.
\end{aligned}
\end{gather*}
Now, we go case by case and see that, if $a\leq b$
\begin{gather*}\label{eq:main11}
\begin{aligned}
	\sum_{\substack{j=0}}^b\binom{b+1}{j}B_jH_{p-1}^{(a-b-1+j)}=&\left(\sum_{\substack{j=0}}^{b-a}+\sum_{\substack{j=b-a+1}}+\sum_{\substack{j=b-a+2\\a\geq2}}^{b}\right)\binom{b+1}{j}B_jH_{p-1}^{(a-b-1+j)}\\\equiv& p\sum_{\substack{j=0}}^{b-a}\binom{b+1}{j}B_{b-a+1-j}B_j+(p-1)\binom{b+1}{a}B_{b-a+1}\\&+p\sum_{\substack{j=b-a+2\\a\geq2}}^{b}\binom{b+1}{j}\dfrac{a-b-1+j}{a-b+j}B_jB_{p-a+b-j}\pmod{p^2},
\end{aligned}
\end{gather*}
if $a=b+1$
\begin{gather*}\label{eq:main12}
\begin{aligned}
	\sum_{\substack{j=0}}^b\binom{b+1}{j}B_jH_{p-1}^{(a-b-1+j)}&\equiv (p-1)+p\sum_{\substack{j=1}}^{a-1}\binom{a}{j}\dfrac{j}{j+1}B_jB_{p-j-1}\pmod{p^2},
\end{aligned}
\end{gather*}
if $a>b+1$
\begin{gather*}\label{eq:main13}
\begin{aligned}
	\sum_{\substack{j=0}}^b\binom{b+1}{j}B_jH_{p-1}^{(a-b-1+j)}&=p\sum_{\substack{j=0}}^{b}\binom{b+1}{j}\dfrac{a-b-1+j}{a-b+j}B_jB_{p-a+b-j}\pmod{p^2}.
\end{aligned}
\end{gather*}
this proves the theorem.

For the more explicit expression of Theorem \ref{th:main1} we will need the following congruence.

\begin{lemma}\label{lemma:kummer}(Kummer's congruence)
	Let $p$ be an odd prime and $b$ an even number such that $p-1$ does not divide $b$. Then
	$$\dfrac{B_{p+b-1}}{p+b-1}\equiv \dfrac{B_b}{b}\pmod{p}.$$
\end{lemma}
Now, using this congruence and properties of Bernoulli numbers we get two corollaries (depending on the parity of $a$ and $b$). 

\begin{corollary}\label{col:nelyginiai}
	For $a,b\in\mathbb{N}$ such that $a+b$ is odd, $a<b-1$ and a prime number $p:\max{(a,b)}+3\leq p$ the following congruence is true
	\begin{align*}
	\sum_{i=1}^{p-1}i^{b}H_i^{(a)}&\equiv \dfrac{1-p}{b+1}\binom{b+1}{a}B_{b-a+1}-\dfrac{p}{b+1}\sum_{\substack{j=0}}^{b}\binom{b+1}{j}B^\#_{a,b,j,p}\pmod{p^2}.
\end{align*}
\end{corollary}

\begin{corollary}\label{col:lyginiai}If conditions of Theorem \ref{th:main1} for $a,b,k$ hold and $a,b$ are of the same parity then
\begin{enumerate}
\item for $0<a=b$
\begin{align*}
	\sum_{i=1}^{p-1}i^{b}H_i^{(a)}&\equiv \left(1-\dfrac{a+2}{a+1}p\right)B_1,
\end{align*}
\item for $0<a<b$
\begin{align*}
	\sum_{i=1}^{p-1}i^{b}H_i^{(a)}&\equiv -\left(1+\dfrac{1}{b+1}\binom{b+1}{a+1}\right)pB_{b-a}B_1\pmod{p^2},
\end{align*}
\item for $0<b<a$
\begin{align}\label{eq:main101}
	\sum_{i=1}^{p-1}i^{b}H_i^{(a)}&\equiv -pB^{\#}_{a,b,1,p}\pmod{p^2}.
\end{align}
\end{enumerate}
\end{corollary}

Notice that (\ref{eq:main101}) has intriguing similarities with the result of Zhao (\ref{eq:Zhao}) for $S_{a,-b,1}$.
\subsection{Proof of Theorem 2}\label{proof2}
We first compute from Proposition \ref{prop:a+b+c} (case $a=c$) and Theorem \ref{th:main1} that 
\begin{gather}\label{eq:mainExp}
\begin{aligned}
	\sum_{i=1}^{p-1}i^{b}\left(H_i^{(a)}\right)^2=&\left(H_{n}^{(a)}\right)^2H_{p-1}^{(-b)}+H_{p-1}^{(2a)}H_{p-1}^{(-b)}-\sum_{i=1}^{p-1}i^{b}H_i^{(2a)}\\&-2\sum_{i=1}^{p-1}i^{-a}H_i^{(a)}H_i^{(-b)}+2\sum_{i=1}^{p-1}i^{b-a}H_i^{(a)}\\\equiv&2\sum_{i=1}^{p-1}i^{b-a}H_i^{(a)}-\sum_{i=1}^{p-1}i^{b}H_i^{(2a)}\\&-\dfrac2{b+1}\sum_{j=0}^b\binom{b+1}{j}(-1)^jB_j\sum_{i=1}^{p-1}i^{b-a+1-j}H_i^{(a)}\pmod{p^2}.
\end{aligned}
\end{gather}
The double sum can be divided into two parts
\begin{align*}\label{eq:doubleSum}	
		\sum_{j=0}^b&\binom{b+1}{j}(-1)^jB_j\sum_{i=1}^{p-1}i^{b-a+1-j}H_i^{(a)}\\&=\left(\sum_{j=0}^{b-a+1}+\sum_{\substack{j=b-a+2\\a\geq2}}^{b}\right)\binom{b+1}{j}(-1)^jB_j\sum_{i=1}^{p-1}i^{b-a+1-j}H_i^{(a)}.
\end{align*}
Using (\ref{eq:Zhao}), we get for $a\geq2$ that
\begin{align*}
	\dfrac{2}{b+1}&\sum_{\substack{j=b-a+2}}^{b}\binom{b+1}{j}(-1)^jB_j\sum_{i=1}^{p-1}i^{b-a+1-j}H_i^{(a)}\equiv p\sum_{\substack{j=b-a+2\\a\geq2}}^{b}\dfrac{1}{j}\binom{b}{j-1}(-1)^jB^\#_{2a,b,j,p}\\&+\dfrac{p}{b+1}\sum_{\substack{j=b-a+2\\a\geq2}}^{b}\binom{b+1}{j}\dfrac{k_{a,a-b-1+j}}{2a-b+j}(-1)^jB_jB_{p-2a+b-j}\pmod{p^2}.
\end{align*}
Note 
\begin{align*}
	\frac1{b+1}\binom{b+1}{j}\dfrac{k_{a,a-b-1+j}}{2a-b+j}&=(-1)^{a+1}\dfrac{b+1-j}{(a+1)(b+1)(a-b+j)}\binom{2a-b+j-1}{a}\binom{b+1}{j}\\&=\dfrac{(-1)^{a+1}}{(a+1)(a-b+j)}\binom{2a-b+j-1}{a}\left[\binom{b+1}{j}-\binom{b}{j-1}\right]\\&=\dfrac{(-1)^{a+1}}{(a+1)(a-b+j)}\binom{2a-b+j-1}{a}\binom{b}{j}\\&=\dfrac{1}{(a+1)(b-a-j)}\binom{b-a-j}{a}\binom{b}{j}\\&=\dfrac{1}{a(a+1)}\binom{b-a-j-1}{a-1}\binom{b}{j},
\end{align*}
hence 
\begin{align*}
	\dfrac{p}{b+1}&\sum_{\substack{j=b-a+2\\a\geq2}}^{b}\binom{b+1}{j}\dfrac{k_{a,a-b-1+j}}{2a-b+j}(-1)^jB_jB_{p-2a+b-j}\\\equiv&\dfrac{p}{a(a+1)}\sum_{j=b-a+2}^b\binom{b-a-j-1}{a-1}\binom{b}{j}B_jB_{p-2a+b-j}\pmod{p^2}.
\end{align*}
As previously, the rest of the proof falls to analysis of the three cases. First we turn to case $b\leq2a-2$. Since $b-a+1-j+a$ is even, when $2|j$, Collorary \ref{col:lyginiai} can be used
\begin{align*}
\dfrac{2}{b+1}&\sum_{j=0}^{b-a}\binom{b+1}{j}(-1)^jB_j\sum_{i=1}^{p-1}i^{b-a+1-j}H_i^{(a)}\\\equiv&\dfrac{p}{b+1}\sum_{\substack{j=0}}^{b-a}\binom{b+1}{j}B^\#_{2a,b,j,p}-\dfrac{p}{b-a+1}\sum_{\substack{j=0}}^{b-a}\binom{b-a+1}{j}B^\#_{2a,b,j,p}\pmod{p^2}
\end{align*}
thus
\begin{align*}
\dfrac2{b+1}&\sum_{j=0}^b\binom{b+1}{j}(-1)^jB_j\sum_{i=1}^{p-1}i^{b-a+1-j}H_i^{(a)}\\\equiv& \dfrac{p}{b+1}\sum_{\substack{j=0}}^{b}\binom{b+1}{j}B^\#_{2a,b,j,p}-\dfrac{p}{b-a+1}\binom{b}{a}B^\#_{2a,b,b-a+1,p}\\&+\dfrac{p}{a(a+1)}\sum_{j=b-a+2}^b\binom{b-a-j-1}{a-1}\binom{b}{j}B_jB_{p-2a+b-j}\\&-\dfrac{p}{b-a+1}\sum_{\substack{j=0}}^{b-a}\binom{b-a+1}{j}B^\#_{2a,b,j,p}\pmod{p^2}.
\end{align*}
Using this and previous identities we get, that, when $2\leq a<b\leq2a-2$
\begin{align*}
	\sum_{i=1}^{p-1}i^{b}\left(H_i^{(a)}\right)^2\equiv&\dfrac{p}{b-a+1}\binom{b}{a}B^\#_{2a,b,b-a+1,p}-\dfrac{p}{b-a+1}\sum_{\substack{j=0}}^{b-a}\binom{b-a+1}{j}B^\#_{2a,b,j,p}\\&-\dfrac{p}{a(a+1)}\sum_{j=b-a+2}^b\binom{b-a-j-1}{a-1}\binom{b}{j}B_jB_{p-2a+b-j}\pmod{p^2}.
\end{align*}
Now we proceed to $b+1=2a$. Here
\begin{align*}
\dfrac{2}{b+1}&\sum_{j=0}^{b-a}\binom{b+1}{j}(-1)^jB_j\sum_{i=1}^{p-1}i^{b-a+1-j}H_i^{(a)}\\=&\dfrac1a\sum_{i=1}^{p-1}i^{a}H_i^{(a)}+\sum_{i=1}^{p-1}i^{a-1}H_i^{(a)}+\dfrac{1}{a}\sum_{\substack{j=2\\2|j\\a\geq3}}^{a-1}\binom{2a}{j}B_j\sum_{i=1}^{p-1}i^{a-j}H_i^{(a)}\\\equiv&\dfrac{1-p}{a}+\dfrac1a\left(1-\dfrac{a+2}{a+1}p\right)-\dfrac{p}{a}\sum_{\substack{j=1}}^{a-1}\binom{a}{j}B^\#_{2a,2a-1,j,p}\\&+\dfrac{p}{2a}\sum_{\substack{j=1}}^{a-1}\binom{2a}{j}\dfrac{j}{j+1}B_jB_{p-2a+j-1}\pmod{p^2},
\end{align*}
and
\begin{align*}
	\sum_{i=1}^{p-1}i^{b}\left(H_i^{(a)}\right)^2\equiv&\dfrac{1-p}{a}-\dfrac{p}{2a(a+1)}-\dfrac{p}{a}\sum_{\substack{j=1}}^{a-1}\binom{a}{j}B^\#_{2a,2a-1,j,p}\\&-\dfrac{p}{a(a+1)}\sum_{j=a+1}^{2a-1}\binom{a-j-2}{a-1}\binom{2a-1}{j}B_jB_{p-j+1}\pmod{p^2}.
\end{align*}
Finally, if $b+1>2a$ (since $b$ is odd $\Rightarrow$ $b\geq2a+1$) one gets
\begin{align*}
\dfrac{2}{b+1}&\sum_{j=0}^{b-a}\binom{b+1}{j}(-1)^jB_j\sum_{i=1}^{p-1}i^{b-a+1-j}H_i^{(a)}\\=&\dfrac{2}{b+1}\left(\sum_{j=0}^{b-2a}+\sum_{j=b-2a+1}+\sum_{j=b-2a+2}^{b-a}\right)\binom{b+1}{j}B_j\sum_{i=1}^{p-1}i^{b-a+1-j}H_i^{(a)}\\\equiv&\sum_{i=1}^{p-1}i^{b-a}H_i^{a}+\dfrac{p}{b+1}\sum_{\substack{j=b-2a+2}}^{b-a}\binom{b+1}{j}B^\#_{2a,b,j,p}\\&+\dfrac{p}{b+1}\sum_{\substack{j=0\\2|j}}^{b-2a}\binom{b+1}{j}\left(1+\dfrac{1}{b-a+2-j}\binom{b-a+2-j}{a+1}\right)B_{b-2a+1-j}B_{j}\\&-\dfrac{1}{b+1}\binom{b+1}{2a}\left(1-\dfrac{a+2}{a+1}p\right)B_{b-2a+1}\pmod{p^2}.
\end{align*}
Therefore
\begin{align*}
\sum_{i=1}^{p-1}i^{b}\left(H_i^{(a)}\right)^2\equiv&\dfrac{B_{b-2a+1}}{b-a+1}\binom{b-a+1}{a}-\dfrac{p}{b-a+1}\sum_{j=0}^{b-2a+1}\binom{b-a+1}{j}B_{b-2a+1-j}B_{j}\\&-\dfrac{p}{b-a+1}\sum_{j=b-2a+2}^{b-a}\binom{b-a+1}{j}B^\#_{2a,b,j,p}\\&-\dfrac{p}{(a+1)(b+1)}\sum_{\substack{j=0}}^{b-2a+1}\binom{b+1}{j}\binom{b-a+1-j}{a}B_{b-2a+1-j}B_{j}\\&-\dfrac{p}{a(a+1)}\sum_{j=b-a+1}^b\binom{b-a-j-1}{a-1}\binom{b}{j}B_jB_{p-2a+b-j}\pmod{p^2}.
\end{align*}
For $b-2a>1$, using Kummer's congruence we can rearrange the result into
\begin{align*}
\sum_{i=1}^{p-1}i^{b}\left(H_i^{(a)}\right)^2\equiv&\dfrac{1-p}{b-a+1}\binom{b-a+1}{a}B_{b-2a+1}-\dfrac{p}{b-a+1}\sum_{j=0}^{b-a}\binom{b-a+1}{j}B^\#_{2a,b,j,p}\\&-\dfrac{p}{(a+1)(b+1)}\sum_{\substack{j=0}}^{b-2a+1}\binom{b+1}{j}\binom{b-a+1-j}{a}B_{b-2a+1-j}B_{j}\\&-\dfrac{p}{a(a+1)}\sum_{j=b-a+1}^b\binom{b-a-j-1}{a-1}\binom{b}{j}B_jB_{p-2a+b-j}\pmod{p^2},
\end{align*}
which completes the proof.
\section{Particular case congruences}\label{particular}
Our results in combination with works of other authors provide several new congruences. Sun, Tauraso \cite{Tauraso2} and Mestrović \cite{Mestrovic} proved that for $p\geq7$
\begin{equation*}
	\sum_{i=1}^{p-1}i^{-1}H_i^{(2)}\equiv
\begin{cases}
B_{p-3}\pmod{p},\\
\dfrac{3}{2p}H_{p-1}^{(2)}\pmod{p^2}.
\end{cases}
\end{equation*}
This and (\ref{eq:a2}) gives us
\begin{equation*}
	\sum_{i=1}^{p-1}\left(H_i^{(2)}\right)^2\equiv
\begin{cases}
-2B_{p-3}\pmod{p},\\
-\dfrac{3}{p}H_{p-1}^{(2)}\pmod{p^2}.
\end{cases}
\end{equation*}
\begin{equation*}
	\sum_{i=1}^{p-1}i\left(H_i^{(2)}\right)^2\equiv
\begin{cases}
B_{p-3}\pmod{p},\\
\dfrac{3}{2p}H_{p-1}^{(2)}+\dfrac p3B_{p-3}\pmod{p^2}.
\end{cases}
\end{equation*}
Using Proposition \ref{prop:a+b+c} (case $c=0$) we can show that for $p\geq 7$
\begin{equation*}
	\sum_{i=1}^{p-1}H_i^{(3)}H_i^{(1)}\equiv
\begin{cases}
-B_{p-3}\pmod{p},\\
\dfrac23pB_{p-3}+\dfrac{3}{2}pH_{p-1}^{(2)}\pmod{p^2}.
\end{cases}
\end{equation*}
Sun \cite{Sun} proved that, for $p>3$
\begin{equation*}
	\sum_{i=1}^{p-1}iH_i^2\equiv1\pmod{p}.
\end{equation*}
Using Corollary \ref{prop:a+b+c} (case $a=b, c=-1$) we get that
\begin{equation*}
	\sum_{i=1}^{p-1}iH_i^2\equiv \dfrac{(p-4)(p-1)}{4}-\dfrac{p^2}{12}B_{p-3}\pmod{p^3}.
\end{equation*}

\end{document}